\documentclass[12pt,leqno]{article}

\usepackage{geometry}

\usepackage{graphicx}
\usepackage{epstopdf,epsfig}

\usepackage[font=small]{caption}
\usepackage{color}

\usepackage{amsmath,amsthm,latexsym,amscd,esint,enumerate}
\usepackage{amsfonts,amssymb}
\usepackage{MnSymbol}

\usepackage[scr=boondoxo,scrscaled=1.05]{mathalfa}

\usepackage{xy}
\xyoption{all}

\usepackage[explicit]{titlesec} 
\titleformat{\section}{\large\scshape\raggedright}{\thesection.}{0.5em}{#1}[\titlerule]
\titlespacing{\section}{0pt}{5pt}{5pt} 
\titleformat{\subsection}[runin]{\scshape}{}{0em}{\underline{#1}}
\titlespacing{\subsection}{0pt}{0pt}{1em}

\swapnumbers

\theoremstyle{plain}
    \newtheorem{theorem}{Theorem}[section]
    \newtheorem{lemma}[theorem]{Lemma}
    \newtheorem*{lemma*}{Lemma}
    \newtheorem*{proposition*}{Proposition}
        \newtheorem*{corollary*}{Corollary}
    \newtheorem*{theorem*}{Theorem}
    \newtheorem{corollary}[theorem]{Corollary}
    \newtheorem{proposition}[theorem]{Proposition}
    
        \newtheorem{question}[theorem]{Question}
\theoremstyle{definition}

    \newtheorem{remark}[theorem]{Remark}
    
\newtheorem{clai}[theorem]{Claim}
\newtheorem*{ack}{Acknowledgments}

\theoremstyle{remark}

\numberwithin{equation}{section}

\newcounter{actr}

\newcounter{bctr}

\newcommand{\Hawaii}{Hawai\kern.05em`\kern.05em\relax i}

\newcommand{\R}{\mathbb{R}}
 
\newcommand{\N}{\mathbb{N}}
\newcommand{\Z}{\mathbb{Z}} 

\newcommand{\HH}{\mathbb{H}}

\DeclareMathOperator{\Iso}{Iso}
\DeclareMathOperator{\sep}{sep}
\DeclareMathOperator{\Aut}{Aut}

\DeclareMathOperator{\PO}{PO}

\DeclareMathOperator{\SL}{SL}
\DeclareMathOperator{\PSL}{PSL}
\DeclareMathOperator{\GL}{GL}

\begin{document}   

\title{On coarse embeddings of amenable groups into hyperbolic graphs}
\date{}

\author{Romain Tessera}

%

\maketitle

\begin{abstract}
In this note we prove that if a finitely generated amenable group admits a regular map to $\HH^n\times \R^d$, then it must be virtually nilpotent of degree of growth at most $d+n-1$.  This is sharp as $\Z^{n+d-1}$ coarsely embeds into $\HH^n\times \R^d$.
We deduce that an amenable group regularly (or coarsely) embeds into a hyperbolic group if and only if it is virtually nilpotent, answering a question of Hume and Sisto. We describe an application to Lorentz geometry due to Charles Frances. 
 \end{abstract}

\section{The result}

It is well-known and easy to see that a finitely generated group that quasi-isometrically embeds into a hyperbolic graph must be hyperbolic. By contrast, Gromov observed that any locally finite graph coarsely embeds into a locally finite hyperbolic graph possibly with unbounded degree: see for instance Groves-Manning's ``Combinatorial horoball" construction \cite[Proposition 3.7.]{GM}. In this note we consider the problem of finding obstructions to the existence of coarse embeddings of finitely generated groups into a {\it bounded degree} hyperbolic graph, and more generally into a direct product of a bounded degree hyperbolic graph with a euclidean space. We shall provide a sharp answer for amenable groups. 

Actually we will consider an even weaker notion of embedding, called {\it regular maps}. 
Recall from  \cite{BST} that a map $\phi:X\to Y$ between the vertex sets of graphs with bounded degree is called regular if there exists a positive integer $m<\infty$ such that $\phi$ is Lipschitz and at most $m$-to-one. The class of regular maps is stable under composition, and includes coarse embeddings.  It is interesting though to note that regular maps form a larger class of maps, allowing some form of ``folding": for instance, the map $n\to |n|$ from $\Z\to \N$ is a regular map but not a coarse embedding. Let $\mathcal{C}$ be the class of  metric spaces which are quasi-isometric to bounded degree graphs. Those include locally compact, compactly generated groups \cite[Proposition 6.7.]{Te}. Following \cite{HMT'}, one can define the class of regular maps between elements of $\mathcal{C}$ as the smallest class of maps containing quasi-isometries, and regular maps between bounded degree graphs\footnote{Alternatively a more explicit definition can be cooked up for a certain class  of metric measure spaces \cite{HMT}.}.

Let $\HH^n$ denote the real hyperbolic space of dimension $n$. The main goal of this note is to establish the following theorem.

\begin{theorem}\label{thm:main}
Let $G$ be a unimodular amenable locally compact compactly generated group. Assume that there exists $n\geq 2$ and $d\in \N$ such that $G$ regularly maps to $\HH^n\times P$, where $P$ is a locally compact, compactly generated group of polynomial growth of degree $d$. Then $G$ must have polynomial growth of degree at most $d+n-1$.
\end{theorem}

To put Theorem \ref{thm:main} in perspective, we recall the following well-known fact.
\begin{proposition}\label{prop:easydirection}
Let $G$ be compactly generated locally compact group of polynomial growth of degree $d$. Then there exists $n=n(d)$ such that $G$ admits a coarse embedding into $\HH^n$.
Moreover, if $G=\R^d$, then it coarsely embeds into $\HH^{d+1}$.
\end{proposition}
\begin{proof}[Proof of Proposition \ref{prop:easydirection}]
By Assouad's embedding theorem \cite{As}, $G$ admits a coarse embedding into $\R^{n-1}$ for some $n\in \N$ only depending on $d$. Define $A_n=\R^{n-1} \rtimes \R$, where $t\in \R$ acts by multiplication by $e^t$. Recall that $A_n$ acts simply transitively by isometries on $\HH^n$. Pulling back the Riemannian metric of $\HH^n$, we obtain a left-invariant Riemannian metric $d$ on $A_n$, such that $(A_n,d)$ is isometric to $\HH^n$. On the other hand $\R^{n-1}$ being a closed subgroup of $A_n$, we deduce that it is coarsely embedded into $\HH^n$. Composing these coarse embeddings we obtain a coarse embedding of $G$ into $\HH^n$ as required.
\end{proof}
 
We deduce from Proposition \ref{prop:easydirection} that $\R^{d+n-1}$ coarsely embeds into $\HH^{n}\times \R^d$, showing that the bound from Theorem \ref{thm:main}  is sharp.

Recall that a graph $Z$ is doubling if there exists $N>0$ such that each ball of radius $2r$ can be covered by $N$ balls of radius $r$.  We can deduce the following more qualitative statement.
 \begin{corollary}\label{cor:main}
Let $G$ be a  unimodular, amenable, locally compact, compactly generated group. The following are equivalent.
\begin{itemize}
\item[(i)] $G$ has polynomial growth;
\item[(ii)] $G$  coarsely embeds into a Gromov hyperbolic finitely generated group $\Gamma$.
\item[(ii')] $G$  regularly maps to a Gromov hyperbolic finitely generated group $\Gamma$.
\item[(iii)] $G$ regularly maps to $X\times Z$, where $X$ is a bounded degree Gromov  hyperbolic graph, and $Z$ is a doubling graph. 
\item[(iv)] $G$ regularly maps to $\HH^n\times \R^n$ for some $n$. 
\end{itemize}
\end{corollary}
 \begin{proof}[Proof of Corollary \ref{cor:main}]
 Clearly, (ii) implies (ii') implies (iii). Assume (i). By Proposition \ref{prop:easydirection}, $G$ coarsely embeds into $\HH^n$ for some $n$, hence in some uniform lattice $\Gamma$ in $\Iso(\HH^n)$, which is hyperbolic. Hence (ii) follows.
 By a celebrated result of Bonk and Schramm \cite{BS}, a hyperbolic graph with bounded degree quasi-isometrically embeds into $\HH^n$ for some large enough $n$. Besides, by Assouad's embedding theorem \cite{As}, any doubling graph can be coarsely embedded into $\R^n$ for some large enough $n$. Hence (iii) implies (iv). Finally (iv) implies (i) is the content of Theorem \ref{thm:main}.
 \end{proof}

The implication (ii) implies (i) for finitely generated amenable groups answers \cite[Question 1.6]{HS}.
Let us mention that (ii) implies (i) (resp.\  (ii') implies (i)) have been proved by Hume and Sisto  \cite{HS} (resp.\ Le Coz and Gournay \cite{CG}) 
under the condition that $G$ is finitely generated {\it and} solvable.

 Our proof of Theorem \ref{thm:main} consists in relating three geometric quantities (see \S \ref{sec:Proof} for definitions): the isoperimetric function (encoding amenability), the separation profile (encoding hyperbolicity), and volume growth (for the conclusion). More specifically, we exploit a simple but subtle relationship between the isoperimetric function and the separation profile discovered by 
Le Coz and Gournay (see Lemma \ref{lem:crucial}). We deduce from it an inequality relating the isoperimetric function and the separation profile which holds when the latter grows sufficiently slowly (see Proposition \ref{prop:sep and j}).  An other crucial ingredient in our proof is Coulhon and Saloff-Coste's famous inequality relating the volume growth and the isoperimetric function. In their proof of (ii') implies (i) for solvable groups, Le Coz and Gournay exploit the calculation of the separation profile of $\HH^n$ by Benjamini, Schramm and Timar \cite{BST}. Similarly, to obtain the quantitative statement of  Theorem \ref{thm:main} (as well as (iv) implies (i) in the corollary), we exploit the recent computation of the separation profile of direct products of the form $\HH^n\times P$ performed by Hume, Mackay and the author in \cite{HMT'}.

 Theorem \ref{thm:main} suggests that the class of finitely generated groups admitting a regular map to $\HH^n\times \R^n$ should be quite restricted. Say that a group is RHP if it is relatively hyperbolic with peripheral subgroups to polynomial growth. Recall that an RHP group admits a coarse embedding into $\HH^n$ for $n$ large enough \cite{DY}. The following question (which we expect to have a negative answer) emerged from a discussion with David Hume and John Mackay.

 \begin{question}
Must a finitely generated group  that regularly (or coarsely) maps to $\HH^n$  be commensurable to a subgroup of a RHP group?  \end{question}

Similarly, one might ask whether a group  that regularly (or coarsely) maps to $\HH^n\times \R^n$ is commensurable to a subgroup of a direct product of an RHP group with a nilpotent group? This turns out to be false as shown by the following example. 
Consider a surface group $\Gamma_g$ of genus $g\geq 2$, and let $\tilde{\Gamma}_g$ be its canonical central extension, obtained by pulling back $\Gamma_g$ in the universal cover of $\SL(2,\R)$. It is well known that $\tilde{\Gamma}_g$ is quasi-isometric to the direct product $\Gamma_g\times \Z$ (see \cite[\S 3.2]{KajalTes}), while $\tilde{\Gamma}_g$ is not commensurable to a direct product with $\Z$ (for instance this can be deduced  from 
\cite[Lemma 3.1]{KajalTes}).

\section{An application to Lorentz geometry (after Charles Frances)}

It is an early observation by Gromov that the isometry group  of a compact Lorentz $(n+1)$-manifold $M_{n+1}$ coarsely embeds into $\HH^n$ (see \cite[Sections 4.1.C and 4.1.D]{Gr} and \cite[Lemma 2.2]{CF} for a detailed proof). 
But it is only in 2021 that this point of view was successfully exploited to yield important information on the structure of $\Iso(M_{n+1})$. Indeed, in a recent breakthrough, Charles Frances managed to prove the following result (see \cite[Theorem B]{CF}).

\begin{theorem}\label{thmCharles1} Let $(M_{n+1}, g)$ be a smooth compact
$(n + 1)$-dimensional Lorentz manifold, with $n \geq 2$.
Then every finitely generated subgroup of $\Iso(M, g)$ either contains a free
subgroup in two generators, in which case it is virtually isomorphic to a discrete subgroup
of $\PO(1, n)$, or is virtually nilpotent of growth degree $\leq n-1$.
\end{theorem}

In \cite{CF}, this result is essentially obtained as a corollary of Theorem \ref{thm:main} and \cite[Theorem A]{CF}, which we recall for convenience:

\begin{theorem}\label{thmCharles2} Let $(M_{n+1}, g)$ be a smooth compact
$(n + 1)$-dimensional Lorentz manifold, with $n \geq 2$. Assume that $\Iso(M, g)$ contains a closed, compactly generated subgroup with exponential growth. Then we are in exactly one of the
following cases:
\begin{enumerate}
\item  The group $\Iso(M, g)$ is virtually a Lie group extension of $\PSL_2(\R)$ by a compact
Lie group.
\item There exists a discrete subgroup $\Lambda\subset \PO(1,d)$ for $2\leq d\leq n$, such that
$\Iso(M, g)$ is virtually a Lie group extension of $\Lambda$ by a compact Lie group.
\end{enumerate}
\end{theorem}

At the time when \cite{CF} was published, Theorem \ref{thm:main} had only been stated (and proved) for finitely generated groups, which is enough for Theorem \ref{thmCharles1}. However, as outilned in \cite{CF}, the full version of Theorem \ref{thm:main} can be used to prove the following result. 
\begin{theorem}\label{thm:CharlesMoi}
Let $(M_{n+1}, g)$ be a smooth compact
$(n + 1)$-dimensional Lorentz manifold, with $n \geq 1$. Then $\Iso(M, g)$ satisfies the Tits alternative.
\end{theorem}
A group is commonly said to satisfy the Tits alternative if any finitely generated subgroup either contains a non-abelian free subgroup, or is virtually solvable. This property refers to Tits' famous theorem according to which such alternative is satisfied by $\GL_n(K)$ for any field  $K$ \cite{Tits}. Since the Tits alternative is obviously stable under central extension, it is also satisfied by connected Lie group (hence by almost connected Lie groups). In the sequel, the connected component of a Lie group $G$ will be denoted by $G^0$.

Theorem \ref{thm:CharlesMoi} is a straightforward consequence of Theorems \ref{thmCharles2} and \ref{thm:main}, and can be proved following the lines of the deduction of \cite[Theorem B]{CF} from \cite[Theorem A]{CF}. For the sake of completness, we provide a proof below.
We shall need the following lemma, which is well-known to experts. 
\begin{lemma}\label{lem:Tits}
Assume that we have an extension of locally compact groups $1\to H\to G\to Q\to 1$, where $H$ is an almost connected Lie group, and $Q$ satisfies the Tits alternative. Then $G$ satisfies the Tits alternative.
\end{lemma}
\begin{proof}
It is easy to check that the Tits alternative is stable under subgroups, direct product, and (as already mentioned) central extension. Let $\Gamma$ be a finitely generated subgroup of $G$. 
Let us  first assume that $H$ is finite. Consider the obvious morphism $p:G\to \Aut(H)\times Q$, with central kernel. We conclude from the fact that $\Aut(H)$ and  $Q$ both satisfy the Tits alternative. 
Let us now turn to  the general case. Since $G/H^0$ is finite-by-$Q$, and therefore satisfies Tits' alternative by the first case, we can  assume that $H$ is connected. Consider the morphism $\pi:G\to \GL(\mathfrak{h}) \times Q$, where the first component is given by the adjoint action of $G$ on the Lie algebra of $H$.  Once again, we conclude from the facts that $\pi$ has central kernel, and  $\GL(\mathfrak{h})$  and $Q$ satisfy the Tits alternative.
\end{proof}

\begin{proof}[Proof of Theorem \ref{thm:CharlesMoi}] Recall that $\Iso(M_{n+1})$ is a Lie group (with possibly infinitely many connected components).
Isometry groups of compact Lorentz surfaces are well-known to be either compact or compact by cyclic, in which case the theorem follows from Lemma  \ref{lem:Tits}. Hence we can assume $n\geq 2$ (as in Theorem \ref{thmCharles2}).
 Let $\Gamma$ be a subgroup of $\Iso(M_{n+1})$, generated by a finite set $S$. Denote by $G$ the closure of $\Gamma$ in $\Iso(M_{n+1})$. Let $K$ be a compact neighborhood of the neutral element in $G$. Note that $K\cup S$ forms a (compact) generating set of $G$. Assume first that $G$ is either non-amenable or non-unimodular. Then in particular it has exponential growth, and we are in the situation of Theorem \ref{thmCharles2}, in which case the conclusion follows from Lemma \ref{lem:Tits}. Otherwise, we combine Gromov's observation (saying that $G$ coarsely embeds into $\HH^n$) with Theorem \ref{thm:main} to deduce that $G$ has polynomial growth. Since $G/G^0$ has polynomial growth, we know by Gromov's theorem \cite{Gro} that $G/G^0$ is virtually nilpotent. In particular $G/G^0$ satisfies the Tits alternative, and we conclude thanks to Lemma \ref{lem:Tits}.
\end{proof}

\begin{remark}
A more careful inspection of the proof yields the stronger conclusion that any closed compactly generated subgroup of $\Iso(M_{n+1})$ is virtually a hyper-central extension of a linear group. In particular, amenable finitely generated subgroups are virtually nilpotent-by-abelian. We leave the details to the reader.
\end{remark}

\section{... and its proof}\label{sec:Proof}

We now turn to the proof of Theorem \ref{thm:main}, which relies on two important quasi-isometry invariants: the isoperimetric function and the separation profile.  We shall use the usual equivalence relation on monotonous functions: $f\approx g$ if $f\preceq g$ and $g\preceq f$, where the latter means that there exists $C>0$ such that $g(n)\leq Cf(Cn)$. The class modulo $\approx$ of a monotonous function will be referred to as its {\it asymptotic behavior}. 

Given a connected graph $\Gamma$, define its isoperimetric function\footnote{We do not use the term isoperimetric profile on purpose, as the latter is usually defined as the inverse of $j_{\Gamma}$.} as follows: \[j_{\Gamma}(n)=\inf_{|A|\leq n} \frac{|\partial A|}{|A|},\]
where $\partial A$ is the set of edges joining vertices of $A$ vertices outside of $A$.

Introduced in  \cite{BST}, the separation profile $\sep_\Gamma(n)$ of a connected graph $\Gamma$ is the supremum over all subgraphs of size $n$, of the number of edges needed to
be removed from the subgraph, in order for the connected pieces of the remaining graph to have size at most $n/2$. 

\begin{remark}
In the original definition of separation profile, the subset which is removed is a subset of vertices. But for bounded degree graphs, it is easy to see that the two resulting definitions of separation profiles have same asymptotic behaviors. Moreover, the advantage of working with edges instead of vertices is that our main intermediate result, Proposition \ref{prop:sep and j}, does not require the assumption that the graph has bounded degree. 
\end{remark}
Now if $X$ is a metric space quasi-isometric to a bounded graph $\Gamma_X$ (e.g.\ $X$ is a compactly generated locally compact group), then we shall define its isoperimetric function and separation profile to be that of $\Gamma_X$. Since their asymptotic behaviors are invariant under quasi-isometry, this definition does not depend on a choice of $\Gamma_X$.

A key result is the following general fact which we believe to be independent interest. Let us say that a graph is amenable if $\lim_{v\to \infty}j_X(v)=0$.

\begin{proposition}\label{prop:sep and j}
Let $X$ be an amenable connected graph. Assume that there exists $a>0$ such that $\sep_X(v)\preceq v^{1-a}$, then $j_X(v)\preceq v^{-a}$.
\end{proposition}
\begin{remark}
Although we will not need it here, the following proof shows that $j_X(v)\preceq \frac{\sep_X(v)}{v}$ for a wide range of behaviour of $\sep_X$, for instance if $\sep_X(v)\approx v^{1-a}(\log v)^b$ for any $a>0$ and $b\in \R$.
\end{remark}
 
Recall that by  Coulhon and Saloff-Coste's  inequality \cite{CoulhonSaloff}, if $G$ is a unimodular amenable locally compact compactly generated group such that $j_G\preceq v^{-a}$, then $V_G(r)\preceq r^{1/a}$. Using that the asymptotic behaviour of the isoperimetric profile and of the separation profile are both invariant under quasi-isometry, we immediately deduce the following result.

\begin{proposition}\label{prop:sep and growth}
Let $G$ be a unimodular amenable locally compact compactly generated group.  Assume that there exists $a>0$ such that $\sep_G(v)\preceq v^{1-a}$. Then $G$ has polynomial growth of degree $\leq 1/a$.
\end{proposition}

We now turn to the proof of Proposition \ref{prop:sep and j}.
Its main ingredient is the following observation which is a variant of \cite[Lemma 3.2.]{CG}. 

\begin{lemma}\label{lem:crucial}
For any graph $X$, and for all $v\geq 2$, one has
\[j_X(3v/4)-j_X(v)\leq \frac{8\sep_X(v)}{v}. \]
\end{lemma}
\begin{proof}
If $j_X(v)=j_X(v/2)$, then by monotonicity of $j_X$, we deduce that $j_X(3v/4)=j_X(v)$, so there is nothing to prove. Else, this implies that there exists a subset $F$ such that $v/2<|F|\leq v$ which realizes the isoperimetric function $j_X(v)$, i.e.\ such that 
$j_X(v)=\frac{|\partial F|}{|F|}.$
By definition of the separation profile, there exists a set $L$ of  edges of the induced graph of $F$, with $|L|\leq \sep_X(v)$, such that after removing them, each connected component has size at most $|F/2|$. 
Order those components $F_1,F_2\ldots, $ such that $|F_1|\geq |F_2|\geq \ldots$, and
let $i$ be the first integer such that $|\bigcup_{j\leq i}F_j|\geq |F|/4$. 
Let $A_1=\bigcup_{j\leq i}F_j$, and let $A_2=\bigcup_{j> i}F_j$. Note that $|F|/4\leq|A_1|\leq |F|/2$, and so $|F|/2\leq |A_2|\leq 3|F|/4$. In particular, for $i=1,2$, we have
\[v/8\leq |A_i|\leq 3v/4\]
For each $i=1,2$, $\partial A_i$ splits into two parts: the edges contained in $L$, and those, denoted by $\partial^e A_i$, contained in $\partial F$. 
We have
\[j_X(v)=\frac{|\partial F|}{|F|}=\frac{|\partial^eA_1|+|\partial^eA_2|}{|A_1|+|A_2|}\geq \min_{i=1,2}\frac{|\partial^eA_i|}{|A_i|}\geq  \min_{i=1,2}\frac{|\partial A_i|-|L|}{|A_i|}.\] 
We deduce that 
\[j_X(v)\geq j_X(3v/4)-\frac{8|L|}{v}\geq j_X(3v/4)-\frac{8\sep_X(v)}{v},\]
from which the lemma follows.
\end{proof}

\begin{proof}[Proof of Proposition \ref{prop:sep and j}]
Let $C\geq 0$ be such that for all $v\geq 1$, $\sep_X(v)\leq  Cv^{1-a}$.
For all $n\geq 1$, we deduce from the lemma that
\[j_X((4/3)^{n-1})-j_X((4/3)^n)\leq \frac{8\sep_X((4/3)^n)}{(4/3)^n}\leq 8C(3/4)^{an}.\]
Summing over $n\geq k+1$, and using that $j_X(v)$ tends to $0$ as $v\to \infty$, we get
\[j_X((4/3)^k)\leq 32(3/4)^{ak}\]
So for all $v$, applying the latter to $k$ for which $(4/3)^{k}\leq v\leq (4/3)^{k+1}$, and using the fact that $j_X$ is non-decreasing, we get
\[j_X(v)\leq  j_X((4/3)^{k})\leq 32(3/4)^{ak}\leq  32(3v/4)^{-a}.\]
Hence $j_X(v)\preceq v^{-a}$, as expected.
\end{proof}

We are now ready to finish the proof of Theorem \ref{thm:main}.

\begin{proof}[Proof of Theorem \ref{thm:main}] 
By \cite[Theorem C]{HMT}, the separation profile of $\HH^n\times P$ satisfies $\sep_{\HH^n\times P}(v)\approx v^{1-\frac{1}{d+n-1}}$ for $n\geq 3$. 
Since the separation profile is monotonous under regular maps \cite{BST,HMT}, we deduce that $\sep_{G}(v)\preceq v^{1-\frac{1}{d+n-1}}$.
Hence by Proposition \ref{prop:sep and growth}, we have $V_G(r)\preceq r^{n+d-1}$, which ends the proof of the theorem.
\end{proof}

\begin{ack} I thank David Hume and John Mackay for their useful comments on a preliminary version of this paper. I am also grateful to Charles Frances for suggesting me to include Theorem \ref{thm:CharlesMoi} and its proof in this note. \end{ack}


\end{document}